\documentclass[]{amsart}
\usepackage{amscd,amsthm,amssymb,amsfonts,amsmath,euscript}

\theoremstyle{plain}
\newtheorem{thm}{Theorem}
\newtheorem{lemma}[thm]{Lemma}
\newtheorem{prop}[thm]{Proposition}
\newtheorem{cor}[thm]{Corollary}

\theoremstyle{definition}

\theoremstyle{remark}
\newtheorem{remark}[thm]{Remark}

\newcommand{\nc}{\newcommand}

\def\makeop#1{\expandafter\def\csname#1\endcsname
  {\mathop{\rm #1}\nolimits}\ignorespaces}
\makeop{Hom}   \makeop{End}   \makeop{Aut}   \makeop{Isom}  \makeop{Pic} 
\makeop{Gal}   \makeop{ord}   \makeop{Char}  \makeop{Div}   \makeop{Lie} 
\makeop{PGL}   \makeop{Corr}  \makeop{PSL}   \makeop{sgn}   \makeop{Spf}
\makeop{Spec}  \makeop{Tr}    \makeop{Nr}    \makeop{Fr}    \makeop{disc}
\makeop{Proj}  \makeop{supp}  \makeop{ker}   \makeop{im}    \makeop{dom}
\makeop{coker} \makeop{Stab}  \makeop{SO}    \makeop{SL}    \makeop{SL}
\makeop{Cl}    \makeop{cond}  \makeop{Br}    \makeop{inv}   \makeop{rank}
\makeop{id}    \makeop{Fil}   \makeop{Frac}  \makeop{GL}    \makeop{SU}
\makeop{Nrd}   \makeop{Sp}    \makeop{Tr}    \makeop{Trd}   \makeop{diag}
\makeop{Res}   \makeop{ind}   \makeop{depth} \makeop{Tr}    \makeop{st}
\makeop{Ad}    \makeop{Int}   \makeop{tr}    \makeop{Sym}   \makeop{can}
\makeop{length}\makeop{SO}    \makeop{torsion} \makeop{GSp} \makeop{Ker}
\makeop{Adm}   \makeop{Mat}   \makeop{nr}  
\def\makebb#1{\expandafter\def
  \csname bb#1\endcsname{{\mathbb{#1}}}\ignorespaces}
\def\makebf#1{\expandafter\def\csname bf#1\endcsname{{\bf
      #1}}\ignorespaces} 
\def\makegr#1{\expandafter\def
  \csname gr#1\endcsname{{\mathfrak{#1}}}\ignorespaces}
\def\makescr#1{\expandafter\def
  \csname scr#1\endcsname{{\EuScript{#1}}}\ignorespaces}
\def\makecal#1{\expandafter\def\csname cal#1\endcsname{{\mathcal
      #1}}\ignorespaces} 

\def\doLetters#1{#1A #1B #1C #1D #1E #1F #1G #1H #1I #1J #1K #1L #1M
                 #1N #1O #1P #1Q #1R #1S #1T #1U #1V #1W #1X #1Y #1Z}
\def\doletters#1{#1a #1b #1c #1d #1e #1f #1g #1h #1i #1j #1k #1l #1m
                 #1n #1o #1p #1q #1r #1s #1t #1u #1v #1w #1x #1y #1z}
\doLetters\makebb   \doLetters\makecal  \doLetters\makebf
\doLetters\makescr 
\doletters\makebf   \doLetters\makegr   \doletters\makegr
     \def\qed{\qedmark\medbreak}%
\def\qedmark{{\enspace\vrule height 6pt width 5pt depth 1.5pt}}%

\normalsize

\makeop{Bl}

\def\Fp{{\bbF}_p}

\newcommand{\Z}{\mathbb Z}
\newcommand{\Q}{\mathbb Q}
\newcommand{\R}{\mathbb R}
\newcommand{\C}{\mathbb C}
\renewcommand{\H}{\mathbb H}  
\newcommand{\A}{\mathbb A}    

\newcommand{\npr}{\noindent }

\newcommand{\isoto}{\stackrel{\sim}{\longrightarrow}}
\nc{\embed}{\hookrightarrow}

\newcommand{\dieu}{Dieudonn\'{e} }

\nc{\ol}{\overline}
\nc{\wt}{\widetilde}
\nc{\opp}{\mathrm{opp}}
\def\ul{\underline}

\def\wh{\widehat}

\makeop{Ram}
\makeop{Rep}

\begin{document}
\renewcommand{\thefootnote}{\fnsymbol{footnote}}
\setcounter{footnote}{-1}
\numberwithin{equation}{section}


\title{Notes on locally free class groups}
\author{Chia-Fu Yu}
\address{
Institute of Mathematics, Academia Sinica and NCTS\\
6th Floor, Astronomy Mathematics Building \\
No. 1, Roosevelt Rd. Sec. 4 \\ 
Taipei, Taiwan, 10617} 
\email{chiafu@math.sinica.edu.tw}
\address{
The Max-Planck-Institut f\"ur Mathematik \\
Vivatsgasse 7, Bonn \\
Germany 53111} 
\email{chiafu@mpim-bonn.mpg.de}


\date{\today} 


\def\Cl{{\rm Cl}}


\begin{abstract}
A theorem of Swan states that 
the locally free class group of a maximal order in a
central simple algebra is isomorphic to a 
restricted ideal class group of the center. 
In this article we discuss this theorem and its generalization to
separable algebras for which it is more applicable to 
integral representations of finite groups. This is an expository
article with aim to introduce the topic for non-specialists.


\end{abstract}
\maketitle


\section{Introduction}
\label{sec:intro}

A more precise title of this article should  be ``Notes on 
locally free class
groups of orders in separable algebras over global fields''. 
Our goal is to introduce the locally free class group of an $R$-order
$\Lambda$ in a separable $K$-algebra. 
Here $R$ is a Dedekind domain
and $K$ is its fraction field, which we shall assume to be a global
field later. We refer to Section~\ref{sec:03} for definitions of separable
$K$-algebras and $R$-orders.
 
The notion of locally free class groups can be defined 
in a more general setting. However, since results discussed here are 
only restricted to the case where the ground field 
$K$ is a global field, we do not attempt to discuss its definition 
as general as it could be. 
Instead we shall illustrate the essential idea of this notion 
(see Section~\ref{sec:cancel}).    

After illustrating the notion, we present the main theorem on 
locally free class groups. 
We then explain 
how to deduce a theorem of Swan 
\cite[Theorem 1, p.~56]{swan:1962} from the main theorem. 
The strong approximation theorem (SAT) plays an important role in the 
proof of the main theorem, and we give a short exposition of the SAT. 
Our another goal to 
take this opportunity to introduce the reader 
(mainly for graduate students and non-specialists) some useful tools 
and results in Algebraic Number Theory and show 
how to apply them together.

\section{The cancellation law}
\label{sec:cancel}

Let us first motivate the notion of locally free class groups by the
classical theorem of Steinitz. Let $R$ be a Dedekind domain with
fraction field $K$ and we will always suppose 
that $R\neq K$. An {\it $R$-lattice} is a
finite $R$-module $M$ which does not have non-zero torsion
elements, that is, $M$ is isomorphic to a finite $R$-submodule in a
(finite-dimensional) $K$-vector space. We have the following results 
concerning the classification of $R$-lattices 
(cf. \cite[Theorem (4.13), p.~85]{curtis-reiner:1}):
\begin{enumerate}
\item Every $R$-lattice $M$ is $R$-projective, and $M\simeq
  \oplus_{i=i}^n J_i$ for some non-zero ideals $J_i$ of $R$, where $n$ is
  the $R$-rank of $M$. 
\item Two $R$-lattices  
$M=\oplus_{i=i}^n J_i$ and $M'=\oplus_{i=i}^m J'_i$
of the form in (1) are isomorphic if and only if $n=m$ and the
products $J_1\cdots J_n$ and $J_1'\cdots J_n'$ are isomorphic. 
\end{enumerate} 

From the statement (2) one can easily deduce the following 
result: If $M$ and $M'$
are two $R$-lattices, then we have
\begin{equation}
  \label{eq:cancel}
  M\oplus R\simeq M'\oplus R \iff M\simeq M' 
\end{equation}
The property (\ref{eq:cancel}) is called 
the {\it cancellation law}. 

As we learned from Algebra, a useful and easier way of studying 
rings to study their modules, 
instead of their underlying ring structures. 
Using this approach, the cancellation law then 
can be used to distinguish certain rings which share the same 
good properties.
For example, consider the quaternion 
$\Q$-algebras $B_{p,\infty}$, which are those 
ramified exactly the two places
$\{p,\infty\}$ of $\Q $, for primes $p$. Choose a maximal order 
$\Lambda(p)$ in each $B_{p,\infty}$, that is, $\Lambda(p)$ is not
strictly contained in another $\Z$-order in $B_{p,\infty}$. 
Then one can show that 
the cancellation law for ideals of $\Lambda(p)$ holds true 
if and only if $p\in 
\{2,3,5,7,13\}$. 
We will also give a proof of this fact later.     

Now let $\Lambda$ be an (not necessarily commutative) 
$R$-algebra which is finitely generated as an $R$-module. 
The above example shows that the cancellation law for (right) 
projective $\Lambda$-modules need not hold. In
Mathematics, we often encounter a situation that a nice 
property we are looking for turns out to be impossible. 
In that situation one usually remedies it by creating a 
more flexible notion so that the desired nice property remains valid 
in a slightly weaker setting. In the present case, one can for example
consider the following weaker equivalence relation: 
\begin{equation}
  \label{eq:weak}
  \text{Define\ } M\sim M' \quad \text{if $M\oplus 
   \Lambda^r\simeq M'\oplus \Lambda^r$
  for some integer $r\ge 0$.} 
\end{equation}
Then it follows from the definition
that the cancellation law holds true for this
new equivalence relation, that is, we have
\begin{equation}
  \label{eq:can1}
  M\oplus \Lambda \sim M'\oplus \Lambda \iff M\sim M'. 
\end{equation}
The modules $M$ and $M'$ satisfying the above property 
are said to be {\it stably isomorphic}. 
The reader familiar with algebraic topology 
would immediately find that the way of defining a ``stable'' notion
here is similar to that  
in the definition of stable homotopy groups. It is also similar to 
that in the definitions of stable freeness and stable rationality. 
These are reminiscent of the definition of groups $K_0$ and $K_1$ in
Algebraic $K$-theory using an inductive limit procedure.


\section{Locally free class groups}
\label{sec:03}

For the rest of this article we assume that the ground field 
$K$ is a global field; that is, $K$ is a finite separable extension of
$\Q$ or $\Fp(t)$. 
Thus, $R$ is the ring of $S$-integers of $K$ for
a finite set $S$ of places which contains all archimedean ones when
$K$ is a number field. 
Let $A$ denote a finite-dimensional separable $K$-algebra. That is,
$A$ is a finite-dimensional semi-simple $K$-algebra such that
the center $C$ of $A$ is a product of finite separable field
extensions $K_i$ of $K$. 
Recall that an {\it $R$-order} in $A$ is 
an $R$-subring of $A$ which is finitely generated as an $R$-module and
generates $A$ over $K$. 
We let $\Lambda$ denote an $R$-order in $A$.
A {\it $\Lambda$-lattice} $M$ is a $R$-torsion free finitely generated
$\Lambda$-module. \\

\npr {\bf Example.} Let $G$ be a finite group with ${\rm char} K \nmid
|G|$. Then the group algebra $A=KG$ is a separable $K$-algebra. 
We can see this by Maschke's Theorem (cf. \cite[Theorem 3.14, p.~42]
{curtis-reiner:1}): Every finite-dimensional 
representation of $G$ over $K$ is a direct sum of irreducible
representations. Then by definition $KG$ is a semi-simple
$K$-algebra. Applying Maschke's Theorem again 
to an algebraic closure $\ol
K$ of $K$, we see that the algebra $\ol K\otimes_K KG=\ol K G$ 
is also semi-simple. 
Therefore, $A$ is a separable $K$-algebra. 
Clearly, the group ring $\Lambda=RG$ is an $R$-order in
$A$, and any representation $M$ of $G$ over $R$ is a $\Lambda$-lattice. \\






For any integer $n\ge 1$, denote by ${\rm LF}_n(\Lambda)$ the set of
isomorphism classes of locally free right $\Lambda$-modules of rank
$n$. Two locally free right $\Lambda$-modules $M$ and $M'$ are said to be
{\it stably isomorphic}, denoted by $M\sim_{s}M'$, if 
\[ M\oplus \Lambda^{r} \simeq M'\oplus \Lambda^r \]
as $\Lambda$-modules for some integer $r\ge 0$. 
The stable class of $M$ will be denoted by $[M]_s$, while 
the isomorphism class is denoted by $[M]$.
By a $\Lambda$-ideal we mean a $\Lambda$-lattice
in $A$, that is, it is an $R$-lattice which is also a $\Lambda$-module. 
Let $\Cl(\Lambda)$ denote the set of 
stable classes of locally free right
$\Lambda$-ideals in $A$. The Jordan-Zassenhaus Theorem
(cf. \cite[Theorem 24.1, p.~534]{curtis-reiner:1}) states that
${\rm LF_1}(\Lambda)$ is a finite set, and hence so 
the set $\Cl(\Lambda)$ is. 
We define the group structure on $\Cl(\Lambda)$ as follows. 
Let $J$ and $J'$ be two locally free
$\Lambda$-ideals. Define
\begin{equation}
  \label{eq:group}
  [J]_s+[J']_s=[J'']_s,
\end{equation}
where $J''$ is any locally free $\Lambda$-ideal satisfying
\begin{equation}
  \label{eq:com}
   J\oplus J'=J''\oplus \Lambda
\end{equation}
as $\Lambda$-modules. Such a $\Lambda$-ideal $J''$ always exists and 
we will see this in Section~\ref{sec:04}.
The following basic lemma shows that 
$\Cl(\Lambda)$ is an abelian group, called 
the {\it locally free class group} of $\Lambda$.

\begin{lemma}
  The finite set $\Cl(\Lambda)$ with the binary operation 
  defined in (\ref{eq:group}) forms an
  abelian group.
\end{lemma}
\begin{proof}
  By (\ref{eq:com}), the commutativity holds true.  
  We prove the associativity.
  Let $J_1,J_2,J_3$ be three locally free ideals of
  $\Lambda$. Suppose we have $[J_1]_s + [J_2]_s=[J']_s$ and $[J']_s+
  [J_3]_s=[J'']_s$. Then  
\[ (J_1 \oplus J_2)\oplus J_3 \simeq \Lambda\oplus J'\oplus J_3\simeq 
 J'' \oplus \Lambda^2. \] 
Similarly if $[J_2]_s + [J_3]_s=[G']_s$ and $[J_1]_s+
  [G']_s=[G'']_s$, then 
$J_1 \oplus (J_2\oplus J_3) \simeq G'' \oplus \Lambda^2$. This shows
$[J'']_s=[G'']_s$ and the associativity holds true. \qed 
\end{proof}
 
We introduce some more notations. Denote by $C$ the center of
$A$. One has $C=\prod^s_i K_i$ and $A=\prod^s_i A_i$, where each 
$A_i$ is a central simple algebra over $K_i$. 
For any place $v$ of $K$, let $K_v$ denote the completion of $K$ at
$v$, and $O_v$ the valuation ring if $v$ is non-archimedean. 
We also write $R_v$ for $O_v$ when  $v\not\in S$. 
Let $A_v:=K_v\otimes_K A$, $C_v:=K_v\otimes_K A$ and 
$\Lambda_v:=R_v\otimes_R \Lambda$ be the
completions of $A$, $C$ and $\Lambda$ at $v$, respectively.
By a place $w$ of $C$ we mean a place $w$ of $K_i$ for some $i$; that the
algebra $A$ splits (resp.~is ramified) at the place $w$ of $C$ 
means that $A_i$ splits (resp.~is ramified) at the place $w$. 
Let $\wh R=\prod_{v\not\in S} R_v$ be the profinite completion of $R$,
and let $\wh K=K\otimes_{R} \wh R$ be the finite 
$S$-adele ring of $K$; one also writes $\A_K^S$ for $\wh K$.  
Put $\wh A:=\wh K \otimes_K A$, $\wh C:=\wh K \otimes_K C$ 
and $\wh \Lambda:=
\wh R\otimes_R \Lambda=\prod_{v\not\in S} \Lambda_v$.

It is a basic fact that 
the set ${\rm LF}_1(\Lambda)$ is isomorphic to 
the double coset space 
$A^\times \backslash \wh A^\times/\wh \Lambda^\times$.   
There is a natural surjective map 
\begin{equation}
  \label{eq:LF1}
  {\rm LF}_1(\Lambda)\to \Cl(\Lambda)
\end{equation}
by sending $[J]\mapsto [J]_s$.  
Let $N_{A_i/K_i}: A_i\to K_i$ denote the reduced norm map. It induces a
surjective map $N_i: \wh A_i^\times \to \wh K^\times_i$ because we
have the surjectivity $A_{i}\otimes K_v\to K_{i}\otimes K_v$ 
for any finite place $v$ of $K$.
The reduced norm map $N:A=\prod_i A_i\to C=\prod_i K_i$ is simply 
defined as the product $N=(N_{A_i/K_i})_i$. Then we have a surjective map 
$N: \wh A^\times\to \wh C^\times$, and it gives rise
to surjective map (again denoted by)
\begin{equation}
  \label{eq:norm}
   N :{\rm LF}_1(\Lambda)\simeq 
A^\times \backslash \wh A^\times /\wh \Lambda^\times \to
N(A^\times)\backslash \wh C^\times /N(\wh \Lambda^\times).
\end{equation}

We will see that $N(A^\times)=C^\times_{+,A}$, where
$C^\times_{+,A}\subset C^\times$ consists of all elements 
$a\in C^\times$ 
with $r(a)>0$ for all real embeddings(places) $r$ 
which are ramified in $A$. The main theorem for 
the locally free class groups is as follows. 

\begin{thm}\label{2}
  The map (\ref{eq:norm}) factors through  ${\rm LF}_1(\Lambda)\to
  \Cl(\Lambda)$ and it 
  induces an isomorphism of finite abelian groups 
  \begin{equation}
    \label{eq:nu}
    \nu: \Cl(\Lambda)\simeq \wh K^\times /C^\times_{+,A}
  N(\wh \Lambda^\times).
  \end{equation}
\end{thm}

We now describe the theorem of Swan on locally free class groups.  
Assume that $A$ is a central simple algebra and $\Lambda$ is a
maximal $R$-order in $A$. 
Define the ray class group $\Cl_A(R)$ of $K$ by
\[ \Cl_A(R):=I(R)/P_A(R), \] 
where $I(R)$ be the ideal group of $R$ and 
$P_A(R)$ be the subgroup generated by the principal ideals $(a)$ for
$a\in K^\times_{+,A}$. Here $K^\times_{+,A}\subset K^\times$ is the 
subgroup of $K^\times$ defined as above. 
In terms of the adelic language, the
group $\Cl_A(R)$ is nothing but the group 
$\wh K^\times/K^\times_{+,A} \wh
R^\times$.  




\begin{thm}[Swan]\label{3}
Let $K$ be a global field and $R$ the ring of $S$-integers of $K$ for
a finite set $S$ of places containing all archimedean ones. 
Let $A$ be a central simple algebra and $\Lambda$ a maximal $R$-order
in $A$. Then theres is an isomorphism of finite abelian groups 
$\Cl(\Lambda)\simeq \Cl_A(R)$.   
\end{thm}

To see Theorem~\ref{3} is an immediate consequence of Theorem~\ref{2},
we just need to check that $N(\Lambda_v^\times)=R_v^\times$ for
$v\not\in S$
($\Lambda_v$ here is a maximal $R_v$-order). 
There exists a maximal subfield $E\subset A_v$ which is unramified over
$K_v$. Since any two maximal orders in $A_v$ are conjugate,
$\Lambda_v$ contains a copy of the ring of integers $O_E$ of $E$.
As $E$ is unramified over
$K_v$, the successive approximation shows that
$N_{E/K_v}(O_E^\times)=R_v^\times$. It follows that
$N(\Lambda_v^\times)=R_v^\times$. 


\begin{prop}\label{4}
  Let $B_{p,\infty}$ and $\Lambda(p)$ for primes $p$ 
  be as in Section~\ref{sec:cancel}. Then the
  cancellation law for ideals of $\Lambda(p)$ holds true if and only
  if $p\in \{2,3,5,7,13\}$.
\end{prop}
\begin{proof}
  The cancellation law holds if and only if the map
  ${\rm LF}_1(\Lambda(p))\to \Cl(\Lambda(p))$ is a bijection. By Swan's
  theorem, the locally free class group $\Cl(\Lambda(p))\simeq
  \Cl_{B_{p,\infty}}(\Z)$ is trivial. Thus, the cancellation law
  holds if and only if the class number
  $h(\Lambda(p))=|{\rm LF}_1(\Lambda(p))|$ is one. On the other hand
  we have the class
  number formula \cite{eichler:1938}
  \begin{equation}
    \label{eq:class_no}
    h(\Lambda(p))=\frac{p-1}{12}+\frac{1}{3}\left
  (1-\left(\frac{-3}{p}\right )\right )+\frac{1}{4}\left
  (1-\left(\frac{-4}{p}\right )\right ),
  \end{equation}
where $\left( \frac{\cdot}{p}\right )$ denotes the Legendre
symbol. From this one easily sees that $h(\Lambda(p))=1$ if 
and only if $p\in \{2,3,5,7,13\}$. \qed  
\end{proof}

For the rest of this section we give a proof of the following basic fact.

\begin{lemma}\label{5}
  Let $A$ is a separable $K$-algebra and $C$ its center.
  Then $N(A^\times)=C^\times_{+,A}$. 
\end{lemma}
\begin{proof}
  Since $A=\prod_i A_i$ and $C^\times_{+,A}=\prod_i
  K^\times_{i,+,A_i}$, it suffices to show $N(A^\times)=K^\times_{+,A}$
  for any central simple $K$-algebra $A$. We can use the
  Hasse-Schilling norm theorem (the local-global principle for the
reduced norm map) to describe $N(A^\times)$:  
\[ N(A^\times)=\{x\in K^\times ; x\in N(A_v^\times) \ \forall\, v\}; \]
see \cite[(32.9) Theorem, p.~275]{reiner:mo} and \cite[(32.20)
Theorem, p.~280]{reiner:mo}). Clearly 
$N(A_v^\times)=K_v^\times$ when $v$ is complex, non-archimedean, or a
real split place for $A$. It remains to show that
if $v$ is a real ramified place for $A$, then 
one has $v(a)>0$ if and only
if $a\in N(A_v^\times)$. 
\end{proof}

\begin{lemma}\label{7}
  Let $\H$ be the real Hamilton quaternion and $n\in \bbN$. Then
  $N(\GL_n(\H))=\R_{+}$. 
\end{lemma}
\begin{proof}
  We give two proofs of this result. One is topological and the other
  one is algebraic. As $\R_+=N(\H^\times)\subset N(\GL_n(\H))$, 
  it suffices to show $N(\GL_n(\H))\subset \R_{+}$.

  (1) The set $\GL_n(\H)^{ss}\subset\GL_n(\H)$
  of semi-simple elements is open and dense in the classical
  topology. By continuity it suffices to show
  $N(x)>0$ for any $x\in \GL_n(\H)^{ss}$. 
  Since such $x$ is contained in a
  maximal commutative semi-simple subalgebra, which is isomorphic to
  $\C^n$, we have $N(x)>0$.   

  (2) The algebraic proof relies on the existence of the \dieu
  (non-commutative) determinant (cf. \cite[p.~165]{curtis-reiner:1}). 
Suppose that $D$ is a central division algebra over any field $K$. 
There is a group homomorphism (called the \dieu determinant) 
\[ \det: \GL_n(D) \to D^\#, \]
where $D^\#=D^\times/[D^\times, D^\times]$. The reduced norm map $N:
D^\times \to K^\times$ gives rise to a map $\nr: D^\# \to K^\times$.
The reduced norm map $N: \GL_n(D) \to K^\times$ can be also described 
as
\[ N (a)= \nr (\det a). \quad \forall \, a\in \GL_n(D). \] 
It follows that $N(\GL_n(D))\subset
N(D^\times)$. Particularly $N(\GL_n(\H))\subset
N(\H^\times)=\R_+$. \qed 
\end{proof}

\begin{remark}
One can show a slightly stronger result that 
the Lie group $\GL_n(\H)$ is connected.
Let $G_1$ be the kernel of the reduced norm
map $N:\GL_n(\H)\to \R^\times$. Then  
$G_1=\ul G_1(\R)$ 
for a connected, semi-simple and simply connected algebraic
$\R$-group  $\ul G_1$,  and hence that $G_1$ is connected. 
Then the fibers of the reduced norm map $N$ are all
connected as they are principal homogeneous spaces under $G_1$.  
We just showed that the image of the map $N$ is also connected
(Lemma~\ref{7}). Thus, $\GL_n(\H)$ is connected.  
\end{remark}



\section{Proof of Theorem~\ref{2}}
\label{sec:04}

For any integer $n\ge 1$ and any ring $L$ not necessarily commutative,
let $\Mat_n(L)$ denote the matrix ring over $L$ and let $\GL_n(L)$
denote the group of units in $\Mat_n(L)$. 
Let $N_n:\Mat_n(A)\to C$ be the reduced
norm map, which induces a surjective homomorphism 
$N_n: \GL_n(\wh A)\to \wh C^\times$. 
For any integer $r\ge 1$, let $I_r\in \Mat_r(\Z)$ be the identity
matrix. Let $\varphi_r:\GL_n \to \GL_{n+r}$ be the morphism of
algebraic groups which sends 
\[ a \mapsto \varphi_r(a)=
\begin{pmatrix}
  a &  \\
    & I_r
\end{pmatrix}. \]

Clearly any locally free right $\Lambda$-module $M$ of rank $n$ 
is isomorphic to a $\Lambda$-submodule in $A^n$. Therefore, the set 
${\rm LF}_n(\Lambda)$ is in bijection with the set of global equivalence
classes of the genus of the standard lattice $\Lambda^n$ in $A^n$. The
latter is naturally isomorphic to $\GL_n(A)\backslash \GL_n(\wh
A)/\GL_n(\wh \Lambda)$. If $n\ge 2$, then it follows from 
the strong approximation theorem 
(see Kneser~\cite{kneser:sa} and  Prasad~\cite{prasad:sa1977}, 
also see Theorem~\ref{SAT}) 
that the induced map
\begin{equation}
  \label{eq:strong}
  N_n: \GL_n(A)\backslash \GL_n(\wh A)/\GL_n(\wh \Lambda)\isoto 
\wh C^\times /N_n(\GL_n(A))N_n(\GL_n(\wh \Lambda))
\end{equation}
is a bijection. 

\begin{lemma}\label{4.1}
  We have 
  \begin{equation}
    \label{eq:reduction}
   \wh C^\times /N_n(\GL_n(A))N_n(\GL_n(\wh \Lambda))=
  \wh C^\times /N(A^\times)N(\wh \Lambda)=\wh C^\times /C^\times_{+,A}
  N(\wh \Lambda).
  \end{equation}
\end{lemma}
\begin{proof}
  We have seen in Lemma~\ref{5} that
  $N_n(\GL_n(A))=N(A^\times)=C^\times_{+,A}$. We now prove
  $N_n(\GL_n(\Lambda_v))=N(\Lambda_v^\times)$ for $v\not \in S$ since
  the statement is local.  
  The group $\GL_n(\Lambda_v)$ contains as
  a subgroup the group $E_n(\Lambda_v)$ of elementary matrices with
  values in $\Lambda_v$. Since $\Lambda_v$ is semi-local, we have a
  result of Bass \cite[Proposition 8.5]{swan:K_order} 
  that $\GL_n(\Lambda_v)$ is generated by the subgroup
  $E_n(\Lambda_v)$ and the image $\varphi_{n-1}(\GL_1(\Lambda_v))$. 
  Since
  $E_n(\Lambda_v)$ is contained in the kernel of $N$,  
  we
  have $N_n(\GL_n(\Lambda_v))=N_n(\varphi_{n-1}(\Lambda_v^\times))=
  N(\Lambda_v^\times)$. \qed
\end{proof}

For any integer $r\ge 1$, 
we say two locally free right $\Lambda$-ideals $J$ and $J'$ are 
{\it $r$-stably isomorphic} 
if $J\oplus \Lambda^r\simeq J'\oplus \Lambda^r$ as
$\Lambda$-modules. Let $\hat c\in \wh A^\times$ be an element such
that $\hat c \Lambda=J$; we have $\varphi_r(\hat c)
\Lambda^{r+1}=J\oplus \Lambda^r$.

The morphism $\varphi_r$ induces the following commutative diagram:
\begin{equation}
  \label{eq:21}
\begin{CD}
  A^\times \backslash \wh A^\times/\wh \Lambda^\times @>\varphi_r>> 
  \GL_{r+1}(A)\backslash \GL_{r+1}(\wh A)/\GL_{r+1}(\wh \Lambda) \\
  @VNVV @V N_{r+1}VV \\
  \wh C^\times/C^\times_{+,A} N(\wh \Lambda^\times) @>{\rm id}>>  \wh
  C^\times/C^\times_{+,A} N(\wh \Lambda^\times), \\
\end{CD}  
\end{equation}
where the reduced norm map $N_{r+1}$ is known be a bijection. 
Two isomorphism classes $[J]$ and $[J']$ in $A^\times \backslash \wh
A^\times/\wh \Lambda^\times$ are $r$-stably isomorphic if and only if
$\varphi_r([J])=\varphi_r([J'])$. As $N_{r+1}$ is an isomorphism, this
is equivalent to $N([J])=N([J'])$. The latter condition 
is independent of $r$. Therefore, we conclude the following statement.

\begin{lemma}\label{4.2} Let $J$ and $J'$ be two locally free right
  $\Lambda$-ideals. The following statements are equivalent.
  \begin{enumerate}
  \item $J$ and $J'$ are stably isomorphic.
  \item $J$ and $J'$ are $r$-stably isomorphic
    for some $r\ge 1$.
  \item $J$ and $J'$ are $r$-stably isomorphic
    for all $r\ge 1$.
  \item One has $N([J])=N([J'])$ in $\wh K^\times/C^\times_{+,A} N(\wh
    \Lambda^\times)$.   
  \end{enumerate}
\end{lemma}

Thus, the reduced norm map $N$ induces an isomorphism
\begin{equation}
  \label{eq:22}
  \nu: \Cl(\Lambda)\simeq \wh C^\times/C^\times_{+,A} N(\wh
\Lambda^\times).
\end{equation}


We now check that $\nu$ is a group
homomorphism. 
Let $J$ and $J'$ be two locally free
$\Lambda$-ideals. Let $\hat c$ and $\hat c'$ be elements in $\hat
A^\times$ such that $\hat c \Lambda=J$ and $\hat c' \Lambda=J'$.  Put
$J'':=\hat c \hat c' \Lambda$. We claim that 
\begin{itemize}
\item[(a)] $J\oplus J'\simeq J''\oplus \Lambda$ as $\Lambda$-modules; 
\item[(b)] $\nu([J]_s)\nu([J']_s)=\nu([J'']_s)$.
\end{itemize}
Statement (a) follows from 
\[ 
\begin{bmatrix}
  \hat c \hat c' & 0 \\
  0 & 1  \\
\end{bmatrix}\cdot \Lambda^2=J''\oplus \Lambda,\quad \text{and} \quad 
N_2\left ( \begin{bmatrix}
  \hat c \hat c' & 0 \\
  0 & 1  \\
\end{bmatrix} \right )=N_2\left ( \begin{bmatrix}
  \hat c  & 0 \\
  0 & \hat c'  \\
\end{bmatrix} \right )
 \]
in $\wh C^\times/C^\times_{+,A} N(\wh \Lambda^\times)$.
Statement (b) follows from
\[ \nu([J]_s)\nu([J']_s)=N([\hat c])N([\hat c'])=N([\hat c \hat
c'])=\nu([J'']_s). \]



This completes the proof of Theorem~\ref{2}. \\

\section{Strong approximation and remarks}
\label{sec:strong_app}
In this supplementary section
we give a short exposition of the strong approximation theorem 
and explain how (\ref{eq:strong}) follows immediately from this. 
We keep the notations of Section~\ref{sec:03}. In particular $K$
denotes a global field and $S$ is a nonempty finite set of 
places of $K$. 

\begin{thm}[The strong approximation theorem]\label{SAT} 
  Let $G$ be a connected, semi-simple and simply connected algebraic
  group over $K$. Suppose that 
  \begin{itemize}
  \item[($*$)] $G$ does not contain any $K$-simple
  factor $H$ such that the topological group 
  $H_S:=\prod_{v\in S} H(K_v)$ is compact.
  \end{itemize}

Then the group $G(K)$ is dense in $G(\A^S_K)$. 
\end{thm}
\begin{proof}
  See Kneser~\cite{kneser:sa} when $K$ is a number field and
  Prasad~\cite{prasad:sa1977} when $K$ is a global function field. The
  results were proved upon the Hasse principle, i.e. the map
  \[ H^1(K,G) \to \prod_{v} H^1(K_v, G) \]
  is injective. The Hasse principle was known to hold for any
  simply-connected group 
  at that time except possibly for those
  of type $E_8$. The last case of type $E_8$ was 
  finally completed by Chernousov in 1989.  \qed   
\end{proof} 
 
The strong approximation theorem is a strong version of
``class number one'' result.  

\begin{cor}\label{11}
  Let $G$ be as in Theorem~\ref{SAT} satisfying the condition ($*$) 
  and assume that $S$ contains all
  archimedean places of $K$. Then for any open compact subgroup
  $U\subset G(\A^S_K)$, the double coset 
  space $G(K)\backslash G(\A^S_K)/U$ consists of a single element.   
\end{cor}

Let $A$, $C$ and $R$ be as in Section~\ref{sec:03}.
Now we let $G$ and $\ul C^\times$ denote the algebraic groups 
$K$ associated to the multiplicative groups of $A$ and $C$, respectively. 
For any commutative $K$-algebra $L$, one has
\[ G(L)=(A\otimes_K L)^\times, \quad 
\ul C^\times(L)=(C\otimes_K L)^\times. \]
We denote again by $N:G\to \ul C^\times$ the homomorphism of algebraic
$K$-groups induced by the reduced norm map $N:A \to C$, and let
$G_1=\ker N$ denote reduced norm-one subgroup. It is easy to see
that the base change $G_1\otimes \ol K$ is a finite product of simple
groups of the form $\SL_{m}$, and hence $G_1$ is semi-simple and
simply connected.   

Recall that $A$ is said to satisfy 
the {\it Eichler condition with respect to $S$}, if for any simple 
factor $A_i$ of $A$ there is one place $w$ of the center $K_i$ 
over some place $v$ in $S$ such that the completion
$A_{i,w}$ at $w$ is not a {\it division} $K_{i,w}$-algebra. 
Another way to rephrase the last condition for $A_i$ is that 
the kernel of the reduced norm map 
\[ N_{A_i/A_i}: \prod_{v\in S} (A_i\otimes K_v)^\times \to \prod_{v\in
  S} (K_i\otimes K_v)^\times \]
is not compact. In other words, the algebra 
$A$ satisfies the Eichler condition
with respect to $S$ (also denote $A$=Eichler$/R$, where $R$ is the
ring of $S$-integers of $K$) if and only if 
the reduced norm-one subgroup $G_1$ satisfies the condition ($*$) in
Theorem~\ref{SAT}. 
In particular, we have the following special case of Theorem~\ref{SAT}
for $G_1$.   

\begin{thm}\label{eichler}
  Let $A$ be a separable $K$-algebra and $G_1$ the associated 
  reduced norm-one
  subgroup defined as above. 
  If $A$ satisfies the Eichler condition with respect to $S$, 
  then $G_1(K)$ is dense in $G_1(\A_K^S)$.  
\end{thm}

Theorem~\ref{eichler} is what we use in the proof of
Theorem~\ref{2}.
When $K$ is a number field, this is the first case of the
strong approximation theorem, proved by Eichler \cite{eichler:sa}.
Swan \cite{swan:1980} gives a more elementary proof of this theorem.   

Suppose that $A$=Eichler$/R$, and let $U$ be an open compact subgroup 
of $G(\A_K^S)=\wh A^\times$. 
We want to show that the induced surjective map 
\begin{equation}
  \label{eq:5.1}
  N:G(K)\backslash G(\A^S_K)/U \to N(G(K))\backslash \wh C^\times
/N(U)
\end{equation}
is injective. Let $\hat c\in \wh C^\times$ be an element and $\hat g\in \wh
A^\times$ with $N(\hat g)=\hat c$. Then the fiber of the class $[\hat
c]$ is  
\begin{equation}
  \label{eq:5.2}
  N^{-1}([\hat c])=G(K)\backslash G(K)G_1(\A_K^S) \hat g U/U.
\end{equation}
If $x_1,x_2\in G_1(A^S_K)$ be two elements, then 
\begin{equation}
  \label{eq:5.3}
  G(K)x_1\hat g U=G(K) x_2 \hat g U \iff 
G_1(K)x_1 (\hat g U \hat g^{-1})=G_1(K)x_2 
(\hat g U \hat g^{-1}).
\end{equation}
Thus, we get a surjective map
\begin{equation}
  \label{eq:5.4}
  G_1(K)\backslash G_1(\A^S_K)/G_1(\A^S_K)\cap \hat g U \hat g^{-1}
\to G(K)\backslash G(K)G_1(\A_K^S) \hat g U/U=N^{-1}([\hat c]).
\end{equation}
As we know the source of (\ref{eq:5.4}) consists of one element 
(Corollary~\ref{11}), the fiber $N^{-1}([c])$ also consists of one
element. This shows that (\ref{eq:strong}) (or (\ref{eq:5.1})) 
is a bijection.

We end with this article by a few remarks.
Theorem~\ref{3} was first proved by Swan 
\cite[Theorem 1, p.~56]{swan:1962}
when $K$ is a number field. Fr\"ohlich \cite[Theorem 2,
p.~118]{frohlich:crelle1975} gave another proof of Swan's theorem
using the ideles. The proof given here is the same as that of 
Fr\"ohlich and of Swan; all uses Theorem~\ref{eichler}. 
The statement for Swan's theorem over global
fields (Theorem~\ref{3}) can be found in 
Curtis-Reiner \cite[Theorem (49.32), p.~233]{curtis-reiner:2} and
Reiner \cite[(35.14) Theorem p.~313]{reiner:mo}. 
Note that in Reiner \cite[(35.14)
Theorem p.~313]{reiner:mo} there is an assumption of
$A$=Eichler$/R$ when $K$ is a global function field, but 
that is superfluous.
Notice that Prasad's theorem, though for most 
general cases, came a few years after Reiner wrote his book {\it
Maximal Orders} (published in 1975). 
This could be the reason why the result \cite[(35.14) Theorem
p.~313]{reiner:mo} is limited to those satisfying the Eichler
condition in the function field case.  
  
The updated version of Swan's Theorem (Theorem~\ref{3}) is
then presented in the later books by Curtis and Reiner. 
They also give a more general variant (Theorem~\ref{2}); see
\cite[(49.17) Theorem, p.~225]{curtis-reiner:2}. The proof of
Theorem~\ref{2} given in Curtis-Reiner \cite{curtis-reiner:2} is 
different from the original proof of Swan and Fr\"olich; it is proved 
based on Algebraic $K$-theory. This of course brings in  more insights
to the topic. Nevertheless, the original proof may be more
accessible for non-specialists as it is much shorter and also conceptual. 
A very nice exposition for the
proof of Theorem~\ref{eichler} 
can be found in Section 51 of
Curtis-Reiner~\cite{curtis-reiner:2}, 
which follows Swan~\cite{swan:1980}. The paper \cite{swan:1980}
contains some minor errors; see \cite[Appendix A]{swan:crelle1983}
for corrections and improvements.

\section*{Acknowledgments}
The manuscript is prepared while the author's stay at 
the Max-Planck-Institut f\"ur Mathematik in Bonn. 
He is grateful to the Institute for kind hospitality
and excellent working environment. The author thanks JK Yu and
Professor Ming-Chang Kang for
pointing out the references \cite{prasad:sa1977} and
\cite{swan:crelle1983} to him.   
The author is partially supported by the grants 
MoST 100-2628-M-001-006-MY4 and 103-2918-I-001-009.



\begin{thebibliography}{10}

\def\jams{{\it J. Amer. Math. Soc.}} 
\def\invent{{\it Invent. Math.}} 
\def\ann{{\it Ann. Math.}} 
\def\ihes{{\it Inst. Hautes \'Etudes Sci. Publ. Math.}} 

\def\ecole{{\it Ann. Sci. \'Ecole Norm. Sup.}}
\def\ecole4{{\it Ann. Sci. \'Ecole Norm. Sup. (4)}}
\def\mathann{{\it Math. Ann.}} 
\def\duke{{\it Duke Math. J.}} 
\def\jag{{\it J. Algebraic Geom.}} 
\def\advmath{{\it Adv. Math.}}
\def\compos{{\it Compositio Math.}} 
\def\ajm{{\it Amer. J. Math.}}
\def\grenoble{{\it Ann. Inst. Fourier (Grenoble)}}
\def\crelle{{\it J. Reine Angew. Math.}}
\def\mrt{{\it Math. Res. Lett.}}
\def\imrn{{\it Int. Math. Res. Not.}}
\def\acad{{\it Proc. Nat. Acad. Sci. USA}}
\def\tams{{\it Trans. Amer. Math. Sci.}}
\def\cras{{\it C. R. Acad. Sci. Paris S\'er. I Math.}} 
\def\mathz{{\it Math. Z.}} 
\def\cmh{{\it Comment. Math. Helv.}}
\def\docmath{{\it Doc. Math. }}
\def\asian{{\it Asian J. Math.}}
\def\jussieu{{\it J. Inst. Math. Jussieu}}
\def\plms{{\it Proc. London Math. Soc.}}

\def\manmath{{\it Manuscripta Math.}} 
\def\jnt{{\it J. Number Theory}} 
\def\ijm{{\it Israel J. Math.}}
\def\ja{{\it J. Algebra}} 
\def\pams{{\it Proc. Amer. Math. Sci.}}
\def\smfmemoir{{\it Bull. Soc. Math. France, Memoire}}
\def\bsmf{{\it Bull. Soc. Math. France}}
\def\sb{{\it S\'em. Bourbaki Exp.}}
\def\jpaa{{\it J. Pure Appl. Algebra}}
\def\jems{{\it J. Eur. Math. Soc. (JEMS)}}
\def\jtokyo{{\it J. Fac. Sci. Univ. Tokyo}}
\def\cjm{{\it Canad. J. Math.}}
\def\jaums{{\it J. Australian Math. Soc.}}
\def\pspm{{\it Proc. Symp. Pure. Math.}}
\def\ast{{\it Ast\'eriques}}
\def\pamq{{\it Pure Appl. Math. Q.}}
\def\nagoya{{\it Nagoya Math. J.}}
\def\forum{{\it Forum Math. }}
\def\tjm{{\it Taiwanese J. Math.}}
\def\rt{{\it Represent. Theory}}
\def\bordeaux{{\it J. Th\'eor. Nombres Bordeaux}}
\def\ijnt{{\it Int. J. Number Theory}}
\def\jmsj{{\it J. Math. Soc. Japan}}


\def\tp{{to appear}}

\newcommand{\princeton}[1]{Ann. Math. Studies #1, Princeton
  Univ. Press}

\newcommand{\LNM}[1]{Lecture Notes in Math., vol. #1, Springer-Verlag}





\bibitem{curtis-reiner:1}
 C.~W.~Curtis and I.~Reiner, {\it  Methods of representation
 theory. Vol. I. With applications to finite groups and orders}. 
 Pure and Applied Mathematics. A Wiley-Interscience Publication. John
 Wiley \& Sons, Inc., New York, 1981, 819 pp.  


\bibitem{curtis-reiner:2}
Charles W.~Curtis and I.~Reiner,  {\it  Methods of representation
theory. Vol. II. With applications to finite groups and orders.} 
Pure and Applied Mathematics (New York). A Wiley-Interscience
Publication. John Wiley \& Sons, Inc., New York, 1987, 951 pp. 

\bibitem{eichler:1938} M. Eichler, \"Uber die Idealklassenzahl total
  definiter Quaternionenalgebren. \mathz~{\bf 43} (1938), 102--109.

\bibitem{eichler:sa} M.~Eichler, Allgemeine
  Kongruenzklasseneinteilungen der Ideale einfacher Algebren \"uber
  algebraischen Zahlk\"orpern und ihre $L$-Reihen, \crelle~{179} (1938)
  227--251.

\bibitem{frohlich:crelle1975} 
A.~Fr\"ohlich, Locally free modules over arithmetic orders. 
\crelle~{\bf 274/275} (1975), 112--124.

\bibitem{kneser:sa} Martin Kneser, 
Strong approximation. {\it Algebraic Groups and Discontinuous
Subgroups} (Proc. Sympos. Pure Math., Boulder, Colo., 1965)
(1966), 187--196, Amer.~Math.~Soc.,~Providence, R.I.

\bibitem{prasad:sa1977} G.~Prasad, 
Strong approximation for semi-simple groups over function fields.
\ann \  (2)~{\bf 105} (1977), no. 3, 553--572. 
 
 












\bibitem{reiner:mo}
I.~Reiner, I. {\it Maximal orders}. London Mathematical Society
Monographs. New Series, {\bf 28}. Oxford University
Press, Oxford, 2003, 395 pp. 




\bibitem{swan:1962} R.~G.~Swan, 
Projective modules over group rings and maximal orders.
\ann \ (2)~{\bf 76} (1962)  55--61. 

\bibitem{swan:K_order} R.~G.~Swan,
{\it K-theory of finite groups and orders.} \LNM{149} (1970),  237 pp.

\bibitem{swan:1980} R.~G.~Swan, Strong approximation and locally free
  modules. Ring theory and algebra, III, pp. 153--223, Lecture Notes in
  Pure and Appl. Math., 55, Dekker, New York, 1980.  

\bibitem{swan:crelle1983}
Richard G.~Swan,  Projective modules over binary polyhedral groups. 
\crelle~{\bf 342} (1983), 66--172.
  













\end{thebibliography}
\end{document}